\documentclass[11pt]{amsart}
\usepackage{amssymb,amsmath,latexsym,enumerate, dsfont}
\usepackage{graphicx,verbatim,enumerate,%
ifpdf,mdwlist}
\usepackage{color}
\usepackage{mathabx}
\ifpdf
\usepackage{hyperref}
\else
\usepackage[hypertex]{hyperref}
\fi
\usepackage{graphicx}
\usepackage{cancel}
\usepackage{float}
\usepackage{tikz-cd}
\renewcommand{\phi}{\varphi}

\newtheorem{theorem}{Theorem}[section]
\newtheorem{fact}[theorem]{Fact}
\newtheorem{lemma}[theorem]{Lemma}
\newtheorem{corollary}[theorem]{Corollary}

\newtheorem{defi}[theorem]{Definition}
\newenvironment{emdef}{\begin{defi} \rm}{ \end{defi}}
\newtheorem{exa}[theorem]{Example}

\newenvironment{remark}{\begin{rem} \rm}{ \end{rem}}

\newtheorem{rem}[theorem]{Remark}

\DeclareMathOperator{\range}{range}

\DeclareMathOperator{\Id}{Id}

\newcommand{\set}[1]{\{{#1}\}}
\newcommand{\rel}[1]{\mathrel{#1}}

\newcommand{\QA}[1]{\forall{#1}\,}
\newcommand{\QE}[1]{\exists{#1}\,}
\newcommand{\ol}[1]{\overline{#1}}

\title[Word problems and ceers]{Word problems and ceers}

\author[V.~Delle Rose]{Valentino Delle Rose}
\address{Dipartimento di Ingegneria Informatica e Scienze Matematiche\\
Universit\`a Degli Studi di Siena\\
I-53100 Siena, Italy}\email{\href{mailto:valentin.dellerose@student.unisi.it}{valentin.dellerose@student.unisi.it}}

\author[L.~San Mauro]{Luca San Mauro}
\address{Institute of Discrete Mathematics and Geometry, Vienna University of Technology, Vienna, Austria}
\email{\href{mailto:luca.san.mauro@tuwien.ac.at}{luca.san.mauro@tuwien.ac.at}}

\author[A.~Sorbi]{Andrea Sorbi}
\address{Dipartimento di Ingegneria Informatica e Scienze Matematiche\\
Universit\`a Degli Studi di Siena\\
I-53100 Siena, Italy}
\email{\href{mailto:andrea.sorbi@unisi.it}{andrea.sorbi@unisi.it}}


\thanks{San Mauro was supported
by the Austrian Science Fund FWF, project M 2461. Sorbi is a member of
INDAM-GNSAGA, and
was partially supported by PRIN 2017 Grant ``Mathematical Logic: models, sets,
computability''.\\
The authors wish to thank an anonymous referee for valuable
comments and corrections which have greatly contributed to improve the paper,
and for having suggested Theorem~\ref{thm:fundamental}, and its consequent
Corollary~\ref{cor:solution}.}


\begin{document}

\begin{abstract}
This note addresses the issue as to which ceers can be realized by word
problems of computably enumerable (or, simply, c.e.) structures (such as
c.e.~semigroups, groups, and rings), where being realized means to fall in
the same reducibility degree (under the notion of reducibility for
equivalence relations usually called ``computable reducibility''), or in
the same isomorphism type (with the isomorphism induced by a computable
function), or in the same strong isomorphism type (with the isomorphism
induced by a computable permutation of the natural numbers). We observe for
instance that every ceer is isomorphic to the word problem of some
c.e.~semigroup, but (answering a question of Gao and Gerdes) not every ceer
is in the same reducibility degree of the word problem of some finitely
presented semigroup, nor is it in the same reducibility degree of some
non-periodic semigroup. We also show that the ceer provided by provable
equivalence of Peano Arithmetic is in the same strong isomorphism type as
the word problem of some non-commutative and non-Boolean c.e.~ring.
\end{abstract}

\maketitle

\section{Introduction}
Computably enumerable equivalence relations, or ceers, have been an active
field of research in recent years.  A great deal of the interest in ceers
certainly is due to the fact that they appear quite often in mathematical
logic (where they appear, for instance, as the relations of provable
equivalence in formal systems), and in general mathematics and computer
science where they appear as word problems of effectively presented familiar
algebraic structures. An important example in this sense is the word problem
for finitely presented (or, f.p.) groups. If $\langle X; R\rangle$ is a
f.p.~group and one codes the universe of the free group $F_X$ on $X$ with
$\omega$, then the \emph{word problem} of the group is the ceer that
identifies two elements $x,y \in F_X$ if $xy^{-1}$ lies in the normal
subgroup of $F_X$ generated by the relators appearing in the presentation $R$
of the group. The word problem of a f.p.~group can be decidable (i.e.~the
corresponding ceer is decidable), but also undecidable, and in fact can be of
any c.e.~Turing degree, or even $m$-degree: this was obtained independently
by Fridman~\cite{Fridman1962}, Clapham~\cite{clapham1964finitely} and
Boone~\cite{Boone1966a,Boone1966b,Boone1971} (despite the difference in
publication dates, the work of these authors was essentially simultaneous).

Of course not every ceer can be the word problem of a f.p.~group, or even of
a computably enumerable (c.e.) group, see Definition~\ref{def:c.e.algebras}
below. For instance, the equivalence classes of the word problem of a
c.e.~group are uniformly computably isomorphic with each other: to show that
the equivalence class of $u$ is isomorphic to the equivalence class of $v$,
just use the mapping $x \mapsto xu^{-1}v$. Therefore no ceer having both
finite classes and infinite classes, or even having at least two classes of
different $m$-degree, can be the word problem of a group. Therefore the
question naturally arises as to which ceers can be identified as word
problems not only of groups, but of other familiar computably enumerable
structures, modulo several ways of ``identifying'' equivalence relations,
based on natural measures of their relative complexity. The present paper is
meant to be a contribution to this line of research.

We first need of course to specify what we mean by ``computably enumerable
structures'' and their ``word problems'', and how we intend to measure the
relative complexity of equivalence relations.

\subsection{C.e.~algebras}
Following the tradition of Mal'cev and Rabin, it is common to postulate that
the complexity of the problem of presenting the particular copy of a
structure is captured by its atomic diagram. Yet, in algebra one naturally
deals with structures whose algebraic structure is easy to describe but it is
hard to know whether two terms represent the same element. The paradigmatic
example of this phenomenon is the construction, independently due to Boone
\cite{Boone1959} and Novikov \cite{Novikov}, of a finitely presented group
with an undecidable word problem. Moreover, the first homomorphism theorem
ensures that every countable algebra arises as the quotient of the term
algebra on countably many generators. So a countable algebra can
\emph{always} be represented in a way in which the complexity of the
structure is  entirely encoded in its word problem. This motivates the idea,
often recurring in the literature, of looking at c.e.~structures as given by
quotienting  $\omega$ modulo a ceer. In this paper, we will only be concerned
with structures that are algebras.

We recall that a \emph{type of algebras} is a set $\tau$ of function symbols,
such that each member  $f \in \tau$ is assigned a natural number $n$, called
the \emph{arity} of $f$. An \emph{algebra of type $\tau$} is a pair
$\mathcal{A}=\langle A, F\rangle$, where $A$ is a nonempty set, and $F$ is a
set of operations on $A$ \emph{interpreting the type}, i.e.~in one-to-one
correspondence with the function symbols in $\tau$, so that $n$-ary function
symbols of $\tau$ correspond to $n$-ary operations in $F$.

\begin{emdef}\label{def:c.e.algebras}
An algebra $\mathcal{A}$ of decidable type $\tau$, is \emph{computably
enumerable} (or, simply, \emph{c.e.}) if there is a triple $\mathcal{A}^-
=\langle \omega, F, E\rangle$ (called  a \emph{positive~presentation of
$\mathcal{A}$}) such that: (1) $F$ consists of computable operations on
$\omega$ interpreting the type $\tau$; (2) $E$ is a ceer,  which is also a
congruence with respect to the operations in $F$; (3) the quotient
$\mathcal{A}_{E}=\langle \omega_{E}, F_{E} \rangle$ (called a \emph{positive
copy of $\mathcal{A}$}) is isomorphic with $\mathcal{A}$, where $F_{E}=
\{f_{E}: f \in F\}$, with $f_{E}([x]_{E})=[f(x)]_{E}$.
\end{emdef}

For a thorough and clear introduction to c.e.~structures see Selivanov's
paper~\cite{Selivanov}, where they are called \emph{positive structures}, and
Koussainov's tutorial~\cite{Bakh}.

We will consider c.e.~algebras $\mathcal{A}=\langle A, F \rangle$ given by
some positive~presentation $\mathcal{A}^-=\langle \omega, F^-, E \rangle$,
and we will work directly with the positive~presentation rather than the
algebra itself. Thus, if we say $a \in A$ we in fact mean any $a^- \in
\omega$ such that $[a^-]_{E}=a$. The ceer $E$ will be often denoted also by
$=_\mathcal{A}$, as it yields equality in the quotient algebra.

\begin{emdef}\label{def:wp0}
The \emph{word problem} of a c.e.~algebra $\mathcal{A}$ is the ceer
$=_\mathcal{A}$.
\end{emdef}

Given a ceer $E$ (having possibly some interesting computational property) it
is natural to ask which algebras $\mathcal{A}$ can be positively presented
having $E$ as their equality relation $=_{\mathcal{A}}$ (see, e.g.,
\cite{gavruskin2014graphs,fokina2016linear}). Surprisingly, much less is
known about the reverse problem, namely, given a class of structures
$\mathfrak{C}$, which ceers are ``realized'' by members of $\mathfrak{C}$?
This is the main topic of our paper. But, of course, we still need to give a
rigorous definition of what we mean by a structure ``realizing'' a ceer.

\subsection{Measures of the relative complexity of equivalence relations}
The most useful and popular way of measuring  the relative complexity of
ceers has been (at least in recent years: e.g., see
\cite{Gao-Gerdes,ceers,joinmeet,fokina2019measuring}) via the following
notion of reducibility.

\begin{emdef}\label{def:reducibility}
Given a pair of equivalence relations $R,S$ on $\omega$ we say that $R$ is
\emph{computably reducible to} $S$  ($R \leq S$) if there exists a computable
function $f$ such that
\[
(\forall x,y)[x \rel{R} y \Leftrightarrow f(x) \rel{S} f(y)].
\]
\end{emdef}
In the rest of the paper ``computable reducibility'' will be simply referred
to as ``reducibility''. This leads to identifying two equivalence relations
$R,S$, if they both belong to the same reducibility \emph{degree}, i.e.\
$R\leq S$ and $S \leq R$.

In this paper we will consider two additional ways of comparing equivalence
relations based on the notion of ``isomorphism''.

If $R$ is an equivalence relation on $\omega$ then for every number $x$ we
denote by $[x]_R$ the $R$-equivalence class of $x$; the collection of all
$R$-equivalence classes is denoted by $\omega_R$.

\begin{emdef}\label{def:isomorphism}
Given ceers $R,S$, we say that $R$ and $S$ are \emph{isomorphic} (notation $R
\simeq S$) if there is a reduction $f:R \rightarrow S$ such that the range of
$f$ intersects all $S$-equivalence classes. We say in this case that $f$
\emph{induces an isomorphism} from $R$ to $S$.
\end{emdef}

The choice of the name ``isomorphism'' is justified by
Lemma~\ref{lem:characterization} below. Following the category theoretic
approach to numberings proposed by Ershov~\cite{Ershov-russian:Book},
equivalence relations on $\omega$ can be structured as objects of a category
(see also \cite{cat-eq-rel}). The lemma shows in fact that, when restricting
attention only to equivalence relations that are ceers, two objects are
isomorphic in the category theoretic sense if and only if they  are
isomorphic in the sense of our Definition~\ref{def:isomorphism}.

\begin{lemma}[Inversion Lemma]\label{lem:characterization}
If $R,S$ are ceers then $f$ induces an isomorphism from $R$ to $S$ if and
only if $f$ has an equivalence inverse, i.e.\ there is a reduction $g:S\leq
R$ such that $g(f(x)) \rel{R} x$, and $f(g(x)) \rel{S} x$, for all $x\in
\omega$.
\end{lemma}

\begin{proof}
By \cite[Lemma~1.1]{joinmeet}.
\end{proof}

\begin{emdef}
We say that $R$ and $S$ are \emph{strongly isomorphic} if there is a
computable permutation $f$ of $\omega$ providing a reduction $f:R \rightarrow
S$. We say in this case that $f$ \emph{induces} the strong isomorphism.
\end{emdef}

Trivially, if $f$ induces a strong isomorphism from $R$ to $S$ then it also
induces an isomorphism from $R$ to $S$.

It is also clear that if $R$ has at least one finite class and $S$ has only
infinite classes then $R$ and $S$ cannot be strongly isomorphic. On the other
hand:

\begin{lemma}\label{lem:all-infinite}
For every ceer $R$ there exists a ceer $S$ having only infinite classes and
such that $R \simeq S$.
\end{lemma}

\begin{proof}
Let $\langle\_,\_\rangle$ be the Cantor pairing function, and let $(\_)_0$ be
its first projection. Given $R$, let $S$ be such that
\[
x \mathrel{S} y \Leftrightarrow (x)_0 \mathrel{R} (y)_0.
\]
As is immediate to see, the computable function $( \_ ) _0$ induces an
isomorphism from $S$ to $R$, since it provides a reduction whose range
intersects all $R$-equivalence classes.
\end{proof}

We summarize the various definitions, and introduce suitable notations for
them.

\begin{emdef}\label{def:comparing-definitions}
If $R,S$ are ceers, we say that
\begin{itemize}
  \item $R$ is \emph{bi-reducible with $S$} (notation $R \equiv S$) if $R
      \leq S$ and $S \leq R$;

   \item $R$ is \emph{isomorphic to $S$} (notation: $R\simeq S$) if there
      is a reduction $f:R \rightarrow S$  such that $\range(f) \cap [x]_{S}
      \ne \emptyset$, for all $x$;

  \item $R$ is \emph{strongly isomorphic to $S$} (notation:
      $R\simeq_{\textrm{s}} S$) if there is a computable permutation
      of $\omega$ reducing $R$ to $S$.
\end{itemize}
\end{emdef}

The following is a useful observation:

\begin{fact}\label{fact:iso-infinite}
If $R,S$ are ceers such that all $R$-classes and all $S$-classes are infinite
then
\[
R \simeq_{\textrm{s}} S \Leftrightarrow R \simeq S.
\]
\end{fact}

\begin{proof}
The nontrivial implication $R \simeq S \Rightarrow R \simeq_{\textrm{s}} S$
follows by a straightforward back-and-forth argument similar to the one used
in the proof of the Myhill Isomorphism Theorem. See for instance \cite[Remark
1.2]{Andrews-Badaev-Sorbi} and \cite[Lemma 2.3]{KazakhPaper}.
\end{proof}

\begin{fact}\label{fact:compact}
The following proper implications hold on ceers $R,S$:
\[
R \simeq_{\textrm{s}} S \Rightarrow R \simeq S \Rightarrow R \equiv S.
\]
\end{fact}

\begin{proof}
The proof follows easily from known facts in the literature, and a few other
obvious observations. We have already observed that any computable function
inducing a strong isomorphism induces also an isomorphism. So
$\simeq_{\mathrm{s}} $ implies $\simeq$. Next, suppose that $f: R\rightarrow
S$ induces an isomorphism. Then $f$ is already a reduction from $R$ to $S$
giving $R\leq S$.  But $f$ has an equivalence-inverse reduction $g: S \leq
R$, so $S\leq R$ as well. In conclusion $R\equiv S$, and thus $\simeq$
implies $\equiv$.

Let us now show through a few examples that the implications are proper. It
is known that there are universal ceers $E$ (which therefore are reducible to
each other: we recall that a ceer $R$ is universal, if $S \leq R$ for every
ceer $S$), whose equivalence classes are all undecidable. This is the case
for instance of precomplete and u.f.p. ceers (for these notions and their
properties,  see for instance the survey paper~\cite{Andrews-Badaev-Sorbi}):
a concrete example of a universal ceer with undecidable equivalence classes
(\cite[Example~2]{BS}) is the ceer $\sim_{PA}$ induced by provable
equivalence of Peano Arithmetic: see  Section~\ref{sct:last}.  Now if $E$ is
such a universal ceer then $E \oplus \Id_{1}$ (that is, the ceer $\{(2x,2y):
x \rel{E} y\} \cup \{(2x+1,2y+1): x,y \in \omega\}$) is universal too, thus
$E \equiv E\oplus \Id_{1}$, but $E \not\simeq E\oplus \Id_{1}$ since that
latter ceer has one decidable equivalence class.

Finally, take $R$ to be any ceer with at least one finite equivalence class,
and let $S$ be the ceer built from $R$ as in the proof of
Lemma~\ref{lem:all-infinite}. Thus $R \simeq S$, but $R
\not{\simeq}_{\mathrm{s}} S$, as $R$ has at least a finite equivalence class
whereas all $S$-equivalence classes are infinite.
\end{proof}

\begin{emdef}
If $R$ is a ceer, $\mathcal{A}$ is a c.e.~algebra, and $\approx\
\in\{\equiv,\simeq, \simeq_{\textrm{s}}\}$ then we say that $R$ is
\emph{$\approx$-realized by $\mathcal{A}$} if $R \approx =_\mathcal{A}$.
(Recall that $=_\mathcal{A}$ denotes the word problem of $\mathcal{A}$.)
\end{emdef}

\begin{emdef}
A class $\mathfrak{C}$ of algebras of the same type is
\emph{$\approx$-complete for a class $\mathbb{C}$ of ceers} (where $\approx\
\in\{\equiv,\simeq, \simeq_{\textrm{s}}\}$) if every ceer in $\mathbb{C}$ is
$\approx$-realized by some c.e.~copy of an algebra from $\mathfrak{C}$. We
simply say that $\mathfrak{C}$ is \emph{$\approx$-complete for the ceers} if
$\mathfrak{C}$ is $\approx$-complete for the class of all ceers.
\end{emdef}

\begin{corollary}\label{cor:ante}
If $\mathbb{C}$  is a class of ceers all of whose members have no finite
equivalence classes, and $\mathfrak{C}$ is a class of algebras then
\[
\mathfrak{C} \textrm{ $\simeq$-complete for $\mathbb{C}$} \Leftrightarrow
\mathfrak{C} \textrm{ $\simeq_{\textrm{s}}$-complete for $\mathbb{C}$}.
\]
\end{corollary}

\begin{proof}
It follows from Fact~\ref{fact:iso-infinite}.
\end{proof}

Moreover, it is trivial to observe:

\begin{fact}
If $\mathbb{C}$ is a class of ceers and $\mathfrak{C}$ is a class of algebras
then
\begin{align*}
&\mathfrak{C} \textrm{ $\simeq$-complete for $\mathbb{C}$} \Rightarrow
\mathfrak{C} \textrm{ $\equiv$-complete for $\mathbb{C}$},\\
&\mathfrak{C} \textrm{ $\simeq_{\textrm{s}}$-complete for $\mathbb{C}$} \Rightarrow
\mathfrak{C} \textrm{ $\simeq$-complete for $\mathbb{C}$}.
\end{align*}
\end{fact}

\begin{proof}
The proof follows from Fact~\ref{fact:compact}.
\end{proof}

\section{Classes of algebras that are complete for the ceers}

We now begin to look at some natural classes of c.e.~algebras in relation to
the problem of $\approx$-completeness for ceers, with $\approx\
\in\{\equiv,\simeq, \simeq_{\textrm{s}}\}$.  Our examples of c.e.~algebras
will be more conveniently introduced via the notion of a computably
enumerable presentation. In a variety of algebras with finite or countable
type, if the\emph{ term algebra}  $T(X)$ on a finite or countable  set $X$
(see e.g. \cite[\textsection10]{Burris-Sankappanavar}) exists (existence is
guaranteed if, as in our future examples, $X$ is nonempty) then, up to
isomorphisms, $T(X)$ can be presented as a computable algebra: we may assume
that $X$ is decidable, $T(X)$ has decidable universe (which is infinite in
all our examples), computable operations, and equality is syntactic equality.
If, in addition the identities of the variety form a c.e.~binary relation on
$T(X)$, then we have the following definition.

\begin{emdef}\label{def:semigroup-cp}
In a variety as above, a \emph{c.e.~presentation} is a pair
$\mathcal{A}=\langle X; R \rangle$ where $X$ is a set, $R$ is a binary
relation on  $T(X)$, and $\mathcal{A}$ denotes the quotient algebra
$T(X)_{/N_{R}}$, where $N_R$ is the c.e.~congruence on $T(X)$ generated by
$R$ together with the identities of the variety.  An algebra $\mathcal{A}$ of
the variety is \emph{c.e.~presented} (\emph{c.e.p.}), if it is of the form
$\langle X; R \rangle$ as just described.

A special case  is provided by
\emph{finite presentations}, where both $X$ and $R$ are finite.
\end{emdef}

The following fact is well known:

\begin{lemma}\label{lem:equiv-wp}
In a variety as above, an algebra is c.e.~if and only if it is isomorphic to
some c.e.p.~algebra.
\end{lemma}

\begin{proof}
We sketch the proof. If $\mathcal{A}=\langle X; R \rangle$ is a
c.e.~presentation, then there is a computable isomorphism $f$ of $T(X)$ with
an algebra having $\omega$ as universe, and equipped with a set $F$ of
suitable computable functions corresponding, via the isomorphism, to the
operations of $\mathcal{A}$. Then $\mathcal{B}=\langle \omega, F, E \rangle$
is a positive presentation of $\mathcal{A}$, where $E$ is the ceer
corresponding under the isomorphism to the c.e.~relation $N_{R}$ on $T(X)$.
Notice that according to Definition~\ref{def:wp0}, equality $=_{\mathcal{B}}$
of $\mathcal{B}$ coincides with $E$.

For the converse, assume that $\mathcal{A}=\langle \omega, F, E\rangle$ is a
positive presentation. By the universal property of $T(\omega)$ (namely, the
term algebra on the set $\omega$ of generators), there is a unique
epimorphism $\nu: T(\omega) \rightarrow \mathcal{A}_{E}$ which commutes with
the mapping $x \mapsto [x]_{E}$ from $\omega$ to $\mathcal{A}_{E}$, and the
insertion of generators $x\mapsto x$  from $\omega$ to $T(\omega)$. Namely,
if  $p(x_{1}, \ldots, x_{k}) \in T(\omega)$ is a term, and $p_{F}$ interprets
$p$ using the operations in $F$, then  $\nu(p(x_{1}, \ldots,
x_{k}))=[p_{F}(x_{1}, \ldots, x_{k})]_{E}$, by the properties of $E$. It
follows that the kernel $R$ of $\nu$  is a c.e.~binary relation on $T(X)$,
and by universal algebra, the c.e.~presentation $\langle \omega; R\rangle$ is
isomorphic with $\mathcal{A}_{E}$.
\end{proof}

To describe some of the consequences of Lemma~\ref{lem:equiv-wp} which are
relevant to our later examples, we first generalize
Definition~\ref{def:isomorphism} to \emph{partial ceers}, i.e.\
c.e.~equivalence relations having as domains c.e.~subsets of $\omega$. If $R,
S$ are partial ceers with domains $X,Y$ respectively, we say that $R$ and $S$
are isomorphic ($R \simeq S$: we use the same symbol as in
Definition~\ref{def:isomorphism}) if there is computable function $g:
X\rightarrow Y$ such that $x \rel{R} y$ if and only if $g(x) \rel{S} g(y)$
for all $x,y \in X$, and $\range(g)$ intersects all $S$-equivalence classes.

One direction of the proof of the previous lemma actually shows that every
c.e.~presentation $\langle X;R\rangle$ has a positive presentation $\langle
\omega, F, E\rangle$ such that $N_{R} \simeq E$ as partial ceers, as
witnessed by the computable isomorphism $f: T(X) \rightarrow \omega$. The
other direction of the proof shows in fact that for every positive
presentation $\mathcal{A}=\langle \omega, F, E\rangle$ there is a
c.e.~presentation $\langle \omega; R\rangle$ which is isomorphic to
$\mathcal{A}_{/E}$, and $R\simeq E$ as partial ceers.  This follows from the
fact that $\nu$ is onto, and therefore the computable mapping $p(x_{1},
\ldots, x_{k}) \mapsto p_{F}(x_{1}, \ldots, x_{k})$ provides a reduction from
$R$ to $E$ whose range intersects all $E$-equivalence classes.

\subsection{The word problem as a ceer on terms, or as a ceer on the free algebra}\label{ssct:modification}
When trying to show that some ceer $S$ is $\simeq$-realized by a
c.e.~presentation $\langle X; R\rangle$, the above remarks suggest, in
accordance to many algebra textbooks  (see e.g.\
\cite[p.252]{Burris-Sankappanavar}) to take $N_{R}$ as the word problem of
the c.e.~presentation, and show  that $S \simeq N_{R}$ as partial ceers. This
is fully consistent with Definition~\ref{def:wp0}, since, as we have seen,
$S\simeq E$, where $E$ is the ceer of the positive presentation assigned to
$\langle X; R\rangle$ in the proof of Lemma~\ref{lem:equiv-wp}.

In fact, our examples of c.e.~algebras will come from varieties (such as
semigroups, monoids, groups, rings) in which the identities of the variety
generate a decidable congruence $I$ on $T(X)$. By decidability of $I$, we
mean that we can fix a computable mapping $p\mapsto \ol{p}: T(X)\rightarrow
T(X)$ with decidable range, picking up exactly one element in each
$I$-equivalence class, so that the free algebra $F(X)$, taken to be
$T(X)_{/I}$, can be presented as a computable algebra having this range as
universe. Let now $\langle X;R\rangle$ be a c.e.~presentation.  By universal
algebra, there is a c.e. congruence $\ol{R}$ on $F(X)$ (namely,
$\ol{R}={N_{R}}_{/I}$, using common notation in universal algebra) such that
$T(X)_{/N_{R}}$ is isomorphic with $F(X)_{/\ol{R}}$ and
\[
p\rel{N_{R}} q \Leftrightarrow \ol{p}\ol{R} \ol{q},
\]
for every $p,q\in T(X)$. This gives an isomorphism of partial ceers between
$N_{R}$ and $\ol{R}$. Conversely, given a binary c.e. relation $\ol{R}$ on
$\ol{X}=\{\ol{x}: x \in X\}$, then one can find a c.e. congruence $R$ on
$T(X)$ such that $T(X)_{/R}$ is isomorphic with $F(X)_{/N_{\ol{R}}}$, where
$N_{\ol{R}}$ is the c.e. congruence on $F(X)$ generated by $\ol{R}$.
Moreover, $R$ and $N_{\ol{R}}$ are isomorphic as partial ceers.

This suggests to adopt, in these varieties, even a more simplified, yet
equivalent, approach to word problems of c.e. algebras, and agree that a
c.e.~presentation is a pair $\langle X;R\rangle$ where $R$ is a binary
c.e.~relation on $F(X)$ and, in this case, $\langle X;R\rangle$ denotes the
quotient $F(X)_{/N_{R}}$, where $N_{R}$ is the congruence generated on $F(X)$
by $R$, and we take $N_{R}$ as the word problem of the c.e.~algebra so
presented. Of course, in general the elements of $F(X)$ will not be presented
directly as certain elements of $T(X)$ but in some simplified ``normal
form'', obtaining in any case a computably isomorphic copy of the free
algebra, and up to isomorphism of partial ceers, the same word problem.

\subsection{Semigroups}
Throughout the paper our references for terminology about semigroups and
monoids are the textbooks \cite{Clifford-Preston1} and \cite{Howie}. In view
of Definition~\ref{def:semigroup-cp} (and the subsequent adjustment in
Subsection~\ref{ssct:modification}), towards an explicit description of a
c.e.p.~semigroup it is sufficient to describe what the free semigroup $F(X)$
on $X$ and the c.e.~binary relation $R$ are. Hence, we recall the free
semigroup on a set $X$ can be taken to be $\langle
X^*\smallsetminus\{\lambda\}, \cdot\rangle$, where in general $Y^*$ denotes
the collection of finite words of letters from a set $Y$, $\lambda$ is the
empty string, and $\cdot$ is a binary operation on words.

\begin{emdef}
A semigroup $S$ is a \emph{right-zero band} if $ab=b$ for all $a,b \in S$.
\end{emdef}

\begin{theorem}\label{thm:fund1}
The class of right-zero bands is $\simeq$-complete for the ceers.
\end{theorem}

\begin{proof}
Let $R$ be a given ceer, and fix a computable set $X=\{x_i: i\in \omega\}$ of
generators. Consider the c.e.~binary relation $\widehat{R}$ on $F(X)$:
\[
 \widehat{R}=\{x_{i} = x_j: i \rel{R} j\} \cup \{x_jx_i=x_i: i,j
 \in \omega \}. \]
Let $S=\langle X; \widehat{R} \rangle$ be the c.e.p. semigroup so presented.
In particular notice that $ux_i=_Sx_i$ for any word $u$ and any generator
$x_i$.

It is easy to see that
\[
i \rel{R} j \Leftrightarrow x_i=_S x_j,
\]
so that $R\leq =_S$ by the reduction $f(i)=x_i$. On the other hand, as the
range of $f$ intersects all $=_S$-equivalence classes (since $u=_S x_i$ where
$x_i$ is the last bit of $U$, as follows from the relations), we have that
$=_S \simeq R$.
\end{proof}

Of course the same result holds if we replace right-zero bands with left-zero
bands.

\subsection{Monoids}
Next, we consider the case of monoids. Recall in this case that the free
monoid $F(X)$ on $X$ can be taken to be $\langle X^*, \cdot\rangle$, where
again $\cdot$ is concatenation.

\begin{emdef}
A monoid $M=\langle M, \cdot\rangle$ is \emph{right-zero band-like} if $ab=b$
for every $a,b \in M \smallsetminus \left\lbrace1 \right\rbrace$ (where $1$
denotes the identity element).
\end{emdef}

\begin{theorem}
The class of right-zero band-like monoids is $\simeq$-complete for the ceers.
\end{theorem}

\begin{proof}
Let $R$ be a given ceer, and fix again a computable set $X=\{x_i: i\in
\omega\}$ of generators. Consider the c.e.~binary relation $\widehat{R}$ on
$F(X)$:
\[
\widehat{R}=\{x_{i} = x_{j}: i \rel{R} j\} \cup \{x_jx_i=x_i: i \in
\omega \smallsetminus \{0\}, \ j \in \omega\} \cup \{x_0 = \lambda\}.
\]
Let $M=\langle X; \widehat{R}\rangle$ be the c.e.p.~monoid so presented. The
proof that $R\simeq =_M$ is as in the proof of Theorem~\ref{thm:fund1}.
\end{proof}

Again, right-zero band-like monoids can be replaced by left-zero band-like
monoids in the result above.

\section{Classes of c.e.~algebras that are not complete for the ceers}

We try in this section to identify algebraic properties that prevent classes
of c.e.~algebras sharing these properties to be $\approx$-complete for the
ceers, with $\approx\ \in\{\equiv,\simeq, \simeq_{\textrm{s}}\}$.

\subsection{Semigroups}
For our first observation, we need the following definition.

\begin{emdef}\label{def:semigroup-periodic}
A semigroup $S$ is \emph{periodic} if, for all $a\in S$, there are numbers $1
\leq n < m$ such that $a^n= a^m$.
\end{emdef}

Recall that a ceer $R$ is \emph{dark} if $R$ has infinitely many equivalence
classes but it does not admit any infinite c.e.~transversal, i.e. an infinite
c.e.~set $W$ such that if $x,y \in W$ and $x \neq y$ then $x \cancel{\rel{R}}
y$. For the existence and properties of dark ceers see~\cite{joinmeet}.

\begin{theorem}\label{thm:non-periodic-semigroups}
The class of semigroups which are not periodic is not $\equiv$-complete for
the ceers.
\end{theorem}

\begin{proof}
Let $S$ be a non-periodic c.e.~semigroup. Then there exists an element $a\in
S$ such that $a^m \ne_S a^n$ if $m \ne n$. Thus $\{a^n: n \in \omega\}$ is an
infinite c.e transversal, implying that $=_S$ cannot $\equiv$-realize any
dark ceer, as the property of having an infinite c.e.~transversal is
invariant under bi-reducibility.
\end{proof}

Recall that a \emph{diagonal function} for an equivalence relation $R$ is a
computable function $d$ such that $d(x) \cancel{\rel{R}} x$, for every $x$.
The next theorem identifies a natural class of semigroups which are not
$\simeq$-complete for the ceers. Examples of semigroups filling the
description in the statement of the theorem are for instance the semigroups
without idempotent elements.

\begin{theorem}\label{thm:semigroups-diagonal}
The class of semigroups $S$ for which there exists a number $n$ such that
$x^{n}\neq_{\textrm{s}} x$ for every $x$ is not $\simeq$-complete for the
ceers.
\end{theorem}

\begin{proof}
Suppose $S$ is a c.e.~semigroup as in the statement of the theorem. Take any
$x$, and define $d(x)=x^{n}$. Then $d$ is a diagonal function for $=_S$, and
thus $S$ cannot $\simeq$-realize any ceer which does not possess a diagonal
function, such as for instance the weakly precomplete ceers (including the
precomplete ones). For these notions and their properties see again the
survey paper~\cite{Andrews-Badaev-Sorbi}.
\end{proof}

\subsection{Monoids}
We now take a quick look at monoids.

\begin{emdef}
Let $M$ be a monoid. A non-unit element $x \in M$ is a \emph{torsion element}
if there exists a number $n>0$ such that $x^n =1$; otherwise $x$ is
\emph{non-torsion}. Moreover, a monoid is said to be \emph{torsion} if every
element is a torsion element, \emph{non-torsion} otherwise.
\end{emdef}

We observe:

\begin{theorem}
The class of non-torsion monoids is not $\equiv$-complete for the ceers.
\end{theorem}

\begin{proof}
Let $x \in M$ be a non-torsion element. Then $\{x, x^2, (x^2)^2, \ldots\}$ is
an infinite c.e.~transversal for $=_M$. The proof is now similar to the proof
of Theorem~\ref{thm:non-periodic-semigroups}.
\end{proof}

\subsection{On finitely presented semigroups and a question of Gao and
Gerdes}

We recall the following theorem from Gao and Gerdes~\cite{Gao-Gerdes} (where
the statement refers to finitely presented groups, but it is obviously
extendable to all groups).

\begin{fact}\label{Gao-Gerdes-on-groups}
The class of groups is not $\equiv$-complete for the ceers.
\end{fact}

\begin{proof}
If $R$ is an undecidable ceer with only finitely many undecidable equivalence
classes (it is easy to see that there are even undecidable ceers with only
finite equivalence classes: for instance, there are dark ceers with only
finite classes, see \cite[Corollary~4.15]{joinmeet}) then there cannot be any
c.e.~group $G$ such that $=_G\equiv R$: for otherwise, by the reduction $=_G
\leq R$ we would have that either $=_{G}$ is finite, and thus $R \nleq
=_{G}$, or there are decidable $=_{G}$-classes, but as observed in the
introduction all $=_{G}$-equivalence classes are computably isomorphic with
each other, which would imply that  $[1]_{G}$ is decidable and thus $=_G$ is
decidable($u =_G v$ if and only if $uv^{-1}\in [1]_{=_G}$), giving that $R$
is decidable by the reduction $R \leq =_G$.
\end{proof}

For this reason, Gao and Gerdes (see \cite[Problem~10.3]{Gao-Gerdes}) ask
whether the class of f.p.~semigroups is $\equiv$-complete for the ceers. This
is an interesting question, motivated by a celebrated theorem  due to
Shepherdson~\cite{Shepherdson} stating that if $\{A_i: i \in \omega\}$ is a
uniformly c.e.~sequence of c.e.~sets (meaning that the relation ``$x \in
A_i$'', in $i,x$, is c.e.), $B$ is a c.e.~set, and the relation ``$x \in
A_i$'' is $\leq_T B$, then there is a f.p.~semigroup $S$ with the following
three properties: (1) there is an effective correspondence $w_{i}\mapsto
A_{i}$ between a c.e.~set $\{w_i: i \in \omega\}$ of words and $\{A_i: i \in
\omega\}$ so that, effectively in $i$, one can find Turing reductions
establishing $[w_i]_{=_S} \equiv_T A_i$; (2) the Turing degrees of the
various classes $[w]_{=_S}$ consist of the least Turing degree, together with
all finite joins of the various degrees $\deg_T(A_i)$ ; (3) $=_S \equiv_T B$.

The next theorem will provide a negative answer to Gao~and~Gerdes' question.

\begin{theorem}\label{thm:fundamental}
Suppose that $\{E_{i}: i \in \omega\}$ is a uniformly c.e.~sequence of ceers
such that the set $\{i: E_i \textrm{ is finite}\}$ is c.e. Then there exists
an infinite ceer $E$ such that for every $i$,
\[
E_{i} \not\leq E \text{ or } E_i \text{ is finite (i.e.~$E_i$ has finitely
many classes).}
\]
In particular, for every $i$, $E_{i} \not\equiv E$.
\end{theorem}

\begin{proof}
Let $V=\{i: \textrm{ $E_{i}$ is finite}\}$ be c.e., and let $\{V_{s}: s \in
\omega\}$  be a c.e.~approximation to $V$, that is, a strong array of finite
sets, with $V_{s}\subseteq V_{s+1}$ for every $s$, and $V=\bigcup_{s}V_{s}$.
Let also $\{\phi_{j}: j \in \omega\}$ be an acceptable indexing of the
partial computable functions.

Our desired $E$  must satisfy the following requirements:
\begin{align*}
&P_{n}: &E \text{ has at least } n+1 \text{ classes}, \\
&Q_{\langle i, j \rangle}: &E_{i} \not\leq E
\text{ via } \varphi_j, \text{ or } E_i \text{ is finite}.
\end{align*}
We order the requirements according to the priority ordering:
\[
P_{0}<Q_{0}<\ldots <P_{n}<Q_{n}<\ldots<
\]
We say that $R$ has \emph{higher priority than} $R'$ (or $R'$ has \emph{lower
priority than} $R$) if $R<R'$.

We construct $E$ in stages. At stage $s$ we define an equivalence relation
$E_{s}$, so that: $E_{0}=\Id$ (the identity ceer); for every $s$,
$E_{s}\subseteq E_{s+1}$, $E_{s}$ is a finite extension of $\Id$ (the
identity ceer) and, uniformly in $s$, $E_{s}\smallsetminus \Id$  can be
uniformly presented by its canonical index; and finally $E= \bigcup_{s}
E_{s}$  is our desired equivalence relation. $E_{s+1}$ will be generated by
$E_{s}$ plus finitely many pairs of numbers which are, we say,
\emph{$E$-collapsed} at $s+1$.

The strategy to satisfy $P_{n}$ consists in picking $n+1$ numbers which are
still pairwise $E$-non-equivalent, and restraining their equivalence classes
from future $E$-collapses.

The strategy to satisfy $Q_{\langle i,j \rangle }$ goes as follows. At a
given stage $s$, we say that \emph{evidence appears that $\phi_{j}$ is not a
reduction from $E_{i}$ to $E$} if one of the following happens:
\begin{enumerate}

\item[(A)] $\phi_{j}$ does not look total, i.e.\ we see some witness $v$
    such that $\phi_{j}$ diverges on $v$;

\item[(B)] we see two witnesses $x,y$ such that $\phi_{j}(x)$ and
    $\phi_{j}(y)$ both converge, and at the given stage
    $x\cancel{\rel{E_{i}}} y$, but already $\phi_{j}(x) \rel{E}
    \phi_{j}(y)$.

\end{enumerate}

Notice that, contrary to what one may expect, we do not bother to seek
evidence given by two witnesses $x,y$ such that $\phi_{j}(x)$ and
$\phi_{j}(y)$ both converge, and at the given stage already $x\rel{E_{i}} y$,
but $\phi_{j}(x) \cancel{\rel{E}} \phi_{j}(y)$. Our action on trying to meet
$Q_{\langle i,j \rangle }$ will force the opponent to give up on totality of
$\phi_{i}$, or leave non-$E_{i}$-equivalent two numbers whose
$\phi_{j}$-images we have already $E$-collapsed.

Notice that, independently of our will, evidence due to (A) may be lost at a
later stage $t$, if $\phi_{j,t}(v)\downarrow$; evidence due to (B) may be
lost at a later stage $t$ if $x\rel{E_{i,t}} y$.

Here is the description of our strategy in isolation:
\begin{enumerate}

\item we wait to see $i \in V$; if $i$ gets enumerated into $V$ then the
    requirement is satisfied, so we stop worrying about it, and
    definitively move on to satisfy the lower priority requirements;

\item while waiting to see $i \in V$ or for evidence to appear that
    $\phi_{j}$ is not a reduction, we threaten to make $E$ finite by
    $E$-collapsing all the elements $\ge m$, where $m$ is a threshold
    indicated to $Q_{\langle i,j \rangle }$ by the restraint placed by
    higher priority requirements;

\item while waiting to see $i \in V$, if evidence has appeared that
    $\phi_{j}$ is not a reduction then

\begin{enumerate}

\item while this evidence persists, we move on to satisfy the lower
    priority requirements;
\item when this evidence gets lost,  we loop back to (2).

\end{enumerate}
\end{enumerate}

The outcomes of the strategy are evident: (1) is a finitary outcome
satisfying the requirement, as $E_{i}$ is finite.

If (1) does not show up, then we claim that we cannot loop between (3b) and
(2) infinitely often. For otherwise $\phi_{j}$ would be total, $E$ finite (as
we $E$-collapse all $x,y>m$), but then $\phi_{j}$ cannot be an injective
reduction from the equivalence classes of $E_{i}$ (which is infinite) to the
equivalence classes of $E$ (which would be finite).  Therefore our strategy
eventually stops at (3b) because of (A) (outcome: $\phi_{j}$ is not total),
or because of (B) (outcome: $x \cancel{\rel{E_{i}}}y$ but $\phi_{j}(x)
\rel{E} \phi_{j}(y)$ for some $x,y$).

Since all strategies have finite outcomes, the conflicts between different
strategies are resolved by a straightforward finite priority argument.

\medskip
\paragraph{\emph{The construction}}
At each stage,  requirements may be initialized, and they are so at stage
$0$; or, in case of $Q$-requirements, they may be declared permanently
satisfied in which case they are met once and for all.

The construction makes use at each stage of the following parameters for
every requirement $R$: if $R$ is initialized, then these parameters are
undefined. The parameter $m^{R}(s)$ denotes the restraint imposed at stage
$s$ by $R$, with $R\in \{P,Q\}$, to lower priority requirements, so that they
can only $E$-collapse pairs of elements $x,y > m^{R}(s)$.

The parameter $M(\langle i,j\rangle,s)$ (if $Q_{\langle i,j\rangle}$ is not
initialized, and thus we may suppose $s>0$), is defined as follows: if there
is $v \leq s$ such that either
\begin{enumerate}
\item  $\phi_{j,s}(v) \!\uparrow$, or
\item $v=\langle x, y\rangle$ and $\phi_{j,s}(x)$ and $\phi_{j,s}(y)$ both
    converge and $x\cancel{\rel{E_{i,s}}} y$, but $\phi_{j,s-1}(x)
    \rel{E_{s}} \phi_{j,s-1}(y)$.
\end{enumerate}
then let $M(\langle i,j\rangle,s)=\langle v,0\rangle$ in the former case,
otherwise $M(\langle i,j\rangle,s)=\langle v,1\rangle$. Let $M(\langle
i,j\rangle,s)=\langle 0,2\rangle$ if there exists no such $v$.

If not otherwise specified, at each stage $s>0$ each parameter maintains the
same value as at the previous stage, or stays undefined if it was undefined
at the previous stage.

We say that $P_{n}$ \emph{requires attention at $s+1$} if it is initialized.

We say that $R=Q_{\langle i,j \rangle}$ \emph{requires attention at $s+1$} if
$Q_{\langle i,j \rangle}$ has not as yet been declared permanently satisfied
and (in order):
\begin{enumerate}
\item $Q_{\langle i,j \rangle}$ is initialized; or
\item $i\in V_{s}$; or
\item $M(\langle i,j\rangle,s+1)\ne M(\langle i,j\rangle,s)$.
\end{enumerate}

\medskip
\paragraph{\emph{Stage $0$}}
Initialize all requirements, and set $m^{R}(k,0)$ and $M(k,0)$ undefined for
all $R \in \{P,Q\}$, and $k\in \omega$. Let $E_{0}=\Id$.

\medskip
\paragraph{\emph{Stage $s+1$}}
Let $R$ be the least requirement that requires attention: there is such a
least requirement since almost all requirements are initialized when we begin
stage $s+1$.

\medskip
\paragraph{\emph{Case $1$}}
If $R=P_{n}$ then $R$ is initialized: pick the least $n+1$ numbers bigger
than any number so far used in the construction (thus these numbers are still
non-$E$-equivalent) and let $m^{R}(s+1)$ be the greatest one of  the numbers
which have been picked; $R$ stops being initialized.

\medskip
\paragraph{\emph{Case $2$}}
Suppose that $R=Q_{\langle i,j\rangle}$. We refer to the various cases for
which $R$ may require attention:
\begin{enumerate}
\item[(a)]  (Case (1) of requiring attention) let
    $m^{R}(s+1)=\max\{m^{R'}(s): R'<R\}$ (notice that no $R'<R$ is
    initialized), so that $R$ stops being initialized;
\item[(b)]  (Case (2) of requiring attention) declare $R$ \emph{permanently
    satisfied} (and will stay so forever);

\item[(c)] (Case (3) of requiring attention) $E$-collapse all $x,y$ such
    that $m^{R}(s+1)< x,y \leq s$;

\end{enumerate}
Whatever the case,  initialize all $R'>R$, by setting
$m^{R'}(s+1)\!\!\uparrow$, and $M(\langle i',j'\rangle,s) \!\uparrow$ if
$Q_{\langle i',j'\rangle}>R$.

Let $E_{s+1}$ be the equivalence relation generated by $E_{s}$ plus the pairs
of numbers which have been $E$-collapsed at $s+1$.

\medskip
\paragraph{\emph{Verification}}
The verification is based on the following lemma.

\begin{lemma}
For every requirement $R$, $R$ is initialized only finitely many times,
$m^{R}=\lim_{s}m^{R}(s)$ exists, $R$ eventually stops requiring attention,
and $R$ is met.
\end{lemma}

\begin{proof}
Suppose that the claim is true of every $R'<R$, and let $s_{0}$ be the
greatest stage at which some $R'<R$ has received attention, with $s_{0}=0$ if
$R=P_{0}$. Let $m=\max\{m^{R'}: R'<R\}$.

At the beginning of stage $s_{0}+1$, $R$ is initialized, and thus requires
attention, acts through (1) or (2a), and after this stage it will never be
initialized again.

\medskip
\paragraph{\emph{Case $R=P_{n}$, for some $n$}}
If $R=P_{n}$, then $R$ acts, picks $n+1$ unused numbers. These numbers are
still $E$-non-equivalent. $R$ defines a value of $m^{R}(s_{0}+1)$ which will
never change hereafter, and thus is the limit value $m^{R}$ of $m^{R}(s)$.
This limit value sets a restraint on lower priority requirements which
therefore can never $E$-collapse any pair of these $n+1$ numbers. This shows
also that $P_{n}$ is met, as the final $E$ has at least $n+1$ equivalence
classes.

\medskip
\paragraph{\emph{Case $R=Q_{\langle i,j\rangle}$, for some $i,j$}}
At stage $s_{0} +1$, $Q_{\langle i,j\rangle}$ defines the last value
$m^{R}=m^{R}(s_{0}+1)$ of its parameter $m^{R}$: notice that this value will
never change again, and is in fact the same as $m^{R'}$, where $R'$ is the
$P$-requirement immediately preceding $R$ in the priority ordering. If $R$
receives attention at some stage $s_{1}+1> s_{0}+1$ and acts through Case
(2b), then the action declares $R$ permanently satisfied, $R$ will never
receive attention again, $E_{i}$ is finite then $R$ is met.

If we exclude action Case (2b) after $s_{0}+1$, then $E_{i}$ is infinite. We
claim that still $R$ requires attention finitely many times after $s_{0}+1$.
For otherwise, at infinitely many stages $s$ we $E$-collapse all numbers
$m^{R}< x,s\leq s$, and therefore $E$ is finite since we $E$-collapse all
$x,y >m^{R}$. On the other hand $E_{i}$ is infinite, so $\phi_{j}$ cannot
induce a $1$-$1$ mapping from $E_{i}$-equivalence classes to $E$-equivalence
classes, thus eventually $M(\langle i,j\rangle, s)$ stabilizes on a value
$\langle v, k\rangle$ with $k\in \{0,1\}$ and stops receiving attention
again: contradiction. So (if we never act through  Case (2b)) we are forced
to conclude that $M(\langle i,j \rangle, s)$ stabilizes on some $\langle v,
0\rangle$, and thus $\phi_{j}$ is not total, $R$ is satisfied, and
$\lim_{s}m^{R}(s)=m$; or it stabilizes on some $\langle \langle x,y\rangle,
1\rangle$, in which case $x \cancel{\rel{E_{i}}} y$ and $\phi_{j}(x) \rel{E}
\phi_{j}(y)$, and $R$ is met.
\end{proof}

\end{proof}

\begin{corollary}
No class $\mathfrak{A}$ of finitely generated semigroups is $\equiv$-complete
for the ceers.
\end{corollary}

\begin{proof}
Up to computable isomorphisms, we can assume that a finitely generated
c.e.p.~semigroup is of the form $\langle \{0, 1, \ldots, n\}, R\rangle$ where
$R$ is a c.e.~subset of $(\{0, 1, \ldots, n\}^{*})^{2}$. Let $f$ be a
computable function such that $\{V_{f(n,i)}: i \in \omega\}$ computably lists
all c.e.~subsets of $(\{0, 1, \ldots, n\}^{*})^{2}$. From this we get a
computable listing $\{S_{\langle n, i \rangle}: n,i\in \omega\}$ (where
$S_{\langle n, i \rangle}=\langle \{0, 1, \ldots, n\}^{*},
V_{f(n,i)}\rangle$) of all finitely generated c.e.p.~semigroups, and a
corresponding computable listing $\{E_{\langle n, i \rangle}: n,i\in
\omega\}$ of their word problems.

In view of the previous theorem it suffices to show that $\{\langle n,
i\rangle: E_{\langle n, i \rangle} \textrm{ finite}\}$ is c.e.\ Let $X=\{0,1,
\ldots, n\}$. We claim that $E_{\langle n, i \rangle}$ is finite if and only
if
\begin{equation}\tag{$\star$}
\QE{m>0}\QA{\sigma\in X^{*}}[|\sigma|=m \Rightarrow
\QE{\tau \in X^{*}}[|\tau|<|\sigma| \,\&\, \tau \rel{E_{\langle n,i\rangle}}
\sigma]],
\end{equation}
which is a c.e.~expression (in which for a given string $\rho$, the symbol
$|\rho|$ denotes the length of $\rho$). On the one hand,  if $E_{\langle
n,i\rangle}$ is finite, one can fix a finite transversal  $A$ which meets all
the equivalence classes of $E_{\langle n,i\rangle}$. Since each word of
$S_{\langle n,i\rangle}$ is equivalent to a word from $A$, we have that
$(\star)$ holds for $m$, where $m=\max\set{|\sigma| : \sigma \in A}+1$.

On the other hand, assume that $(\star)$ holds, and fix such an $m$. We claim
in this case that every word is $E_{\langle n,i\rangle}$-equivalent to some
word of length $\leq m$. Towards a contradiction, let $n> m$ be the least
number such that there exists $\sigma$ with $|\sigma|=n$, and
$[\sigma]_{E_{\langle n,i\rangle}}$ contains no words of length $\leq m$.
Now, let $\sigma_{0}=\sigma \restriction m$ (i.e., the initial segment of
$\sigma$ of length $m$), and let $\sigma_{1}$ be such that
$\sigma=\sigma_{0}\sigma_{1}$. Then $\sigma_{0}$  is $E_{\langle
n,i\rangle}$-equivalent to some $\rho$  with $|\rho|<m$. Therefore, by
definition of  $\langle X; E_{\langle n,i\rangle}\rangle$, we have that
$\sigma =\sigma_{0}\sigma_{1}E_{\langle n,i\rangle} \rho\,\sigma_{1}$, but
$|\rho\,\sigma_{1}|<n$,  contradicting the minimality of $n$.
\end{proof}

As a particular case, this provides a negative solution to Gao~and~Gerdes'
question:

\begin{corollary}\label{cor:solution}
The class of f.p.~semigroups is not $\equiv$-complete for the ceers.
\end{corollary}

\begin{proof}
Immediate.
\end{proof}

Next, we observe that, given a f.p.~semigroup $S$, the number of finite and
infinite equivalence classes of the word problem $=_S$ gives us some
information about the ceers realized by $S$. We basically owe the following
arguments to \cite{Bokut-Kukin}, see also~\cite{Litvinceva}.

\begin{lemma}\label{lem:tree-like}
If $S$ is a f.p.~semigroup then there is a partial computable function $\psi$
such that for every word $w$, $\psi(w)\downarrow$ if and only if the
$=_S$-equivalence class of $w$ is finite, and, when convergent, $\psi(w)$
outputs the canonical index of the equivalence class $[w]_{=_S}$ of $w$.
\end{lemma}

\begin{proof}
Given a word $w$, we can effectively generate its $=_S$-equivalence class in
a treelike fashion as follows. The root of the tree is $w$. Each node $u$ has
as children the words that can be obtained from $u$ using the relations and
which have not yet appeared as a node in the path from the root to the
present node. Note that one relation produces only finitely many children,
and there are only finitely many relations: hence, this is a finitely
branching tree. By the K\"onig Lemma if the equivalence class of $w$ is
finite, we eventually stop generating new nodes on any branch of the tree:
when this happen we have generated the entire equivalence class of $w$, and
we can compute the canonical index of this class.
\end{proof}

\begin{theorem}
Let $S$ be a f.p.~semigroup. 	
\begin{enumerate}[(i)]
\item If $S$ has finitely many infinite equivalence classes, then $=_S$ is
    decidable. Therefore no undecidable ceer can be $\equiv$-realized by
    such an $S$.
\item If $S$ has infinitely many finite equivalence classes, then $=_S$ is
    light. Therefore, neither finite nor dark ceers can be
    $\equiv$-realized by such an $S$.
\end{enumerate}
\end{theorem}

\begin{proof}
Suppose that $=_{S}$ has only finitely many infinite equivalence classes.
Assume that $\{v_i: i \in I\}$ is a finite set of words, with $v_i \neq_S
v_j$ if $i \ne j$, spanning these infinite equivalence classes. Given words
$x,y$, generate the equivalence classes of $x$ and $y$ in a tree-like fashion
as in the proof of Lemma~\ref{lem:tree-like}, until one of the following
happens: (1) $[x]_{=_{S}}$ and $[y]_{=_{S}}$ cannot grow any more (and we can
decide this, as explained in the proof Lemma~\ref{lem:tree-like}); (2) some
$v_i$ is generated in one equivalence class, and some $v_j$ is generated in
the other one; (3) some $v_i$ is generated in one of the two equivalence
classes and the other one has stopped (again, we can decide this latter
outcome). In any case we can decide if the two words are equal. This proves
statement (i).

Now, we prove (ii). Let $S$ be a f.p.~semigroup with infinitely many finite
equivalence classes and let $\psi$ be the partial computable function of
Lemma~\ref{lem:tree-like}. Using $\psi$ we can build in stages an infinite
c.e.~transversal $\{a_0, a_1, \dots\}$ for $=_S$:
	
\medskip
	
\paragraph{\emph{Step $0$}}
Let $w_0$ be the first word such that $\psi(w_0)\downarrow$ and define $a_0$
to be the least element of the finite set $D_{\psi(w_0)}$.
	
\medskip
	
\paragraph{\emph{Step $n+1$}}
Let $w_{n+1}$ be the first word such that $\psi(w_{n+1}) \downarrow$ and
\[
D_{\psi(w_{n+1})} \cap (\bigcup_{i\le n} D_\psi(w_{i}))=\emptyset
\]
and let $a_{n+1}$ be the least element of $D_{\psi(w_{n+1})}$.
\end{proof}

\section{Classes of algebras $\simeq_{\mathrm{s}}$-realizing provable equivalence of Peano Arithmetic}\label{sct:last}

Although by Fact~\ref{Gao-Gerdes-on-groups} there are ceers $R$ such that $R
\rel{ \not{\equiv}} =_G$, for every c.e.~group $G$, it is known that there
are f.p.~groups $G$ such that $=_G$ is universal. This was first proved by
Miller~III~\cite{MillerIII-decision}. Another example, due to
\cite{Nies-Sorbi} refers to the computability theoretic notion of effective
inseparability. We recall that a disjoint pair $(U,V)$ of sets of numbers is
\emph{effectively inseparable} (\emph{e.i.}) if there exists a partial
computable function $\psi$ such that for each pair $(u,v)$, if $U \subseteq
W_{u}$  and $V \subseteq W_{v}$ and $W_{u} \cap W_{v}=\emptyset$ then
$\psi(u,v)$ converges and $\psi(u,v) \notin W_{u} \cup W_{v}$.  A f.p.~group
$G$ is built in \cite{Nies-Sorbi} such that $=_G$ is \emph{uniformly
effectively inseparable} i.e.~uniformly in $x,y$ one can find an index of a
partial recursive function $\psi$ witnessing that the pair of sets
$([x]_{=_{G}}, [y]_{=_{G}})$ is e.i., if $[x]_{=_{G}} \cap [y]_{=_{G}} =
\emptyset$. Such a f.p.~group has universal word problem, since it is known
(\cite{ceers}) that every uniformly effectively inseparable ceer is
universal.

An important $\simeq_{\mathrm{s}}$-type among the universal ceers is given by
the $\simeq_{\mathrm{s}}$-type of the relation $\sim_T$ of provable
equivalence of any consistent formal system $T$ extending Robinson's system
$Q$ (see for instance Smorynski~\cite{Smorynski-logical} for an introduction
to formal systems of arithmetic), i.e. $x \sim_T y$ if (identifying sentences
with numbers through a suitable G\"odel numbering) $T\vdash x \leftrightarrow
y$. For example, let us take $T$ to be Peano Arithmetic.

The question naturally arises as to which algebras
$\simeq_{\mathrm{s}}$-realize $\sim_T$. Notice that by
Fact~\ref{fact:iso-infinite}, ``$\simeq_{\mathrm{s}}$-realizing $\sim_T$'' is
equivalent to ``$\simeq$-realizing $\sim_T$''. Here are some initial remarks
about this question:

\begin{enumerate}
  \item As far as we know, the question of whether there are
      f.p.~semigroups, or f.p.~groups, having word problems strongly
      isomorphic to $\sim_T$ is still open.
  \item On the other hand, by Theorem~\ref{thm:fund1} there exist
      c.e.~semigroups whose word problem is strongly isomorphic to
      $\sim_T$. We do not know if there are c.e.~groups
      $\simeq_{\mathrm{s}}$-realizing $\sim_T$.
  \item If one computably identifies with numbers the sentences of our
      chosen formal system $T$, and considers the computable operations
      provided by the connectives $\wedge$, $\lor$, $\neg$, $\bot$, $\top$
      (where $\bot$ and $\top$ denote any contradiction and any theorem,
      respectively), then
   \[
   \langle \omega, \wedge, \lor, \neg, \bot, \top, E \rangle
   \]
   (where $x \rel{E} y$ if $T \vdash x \leftrightarrow y$) is a positive
   presentation of the Lindenbaum algebra of the sentences of $T$, which is
   therefore a c.e.~Boolean algebra. It is known that the word problem of
   this c.e.~Boolean algebra is strongly isomorphic to $\sim_T$: see
\cite{Pour-El-Kripke} (see also \cite{Montagna-Sorbi:Universal}).
\end{enumerate}

The above item (3) identifies a very special class of c.e.~rings which
$\simeq_{\mathrm{s}}$-realize $\sim_T$, namely Boolean rings, i.e.~rings
satisfying $x^2=x$ for all $x$. Is that all? Can we find non-Boolean
c.e.~rings $\simeq_{\mathrm{s}}$-realizing $\sim_T$? We will identify in the
following a c.e.~ring $R$ which is neither Boolean nor commutative, such that
$=_R\simeq_{\mathrm{s}} \sim_T$.

\smallskip

The $\simeq_{\mathrm{s}}$-type of $\sim_T$ can be characterized through the
already given notion of a diagonal function, and the notion of
\emph{uniformly finite precompleteness}, due to \cite{Montagna:ufp}.

\begin{emdef}\label{def:ufp}
A nontrivial ceer $S$ is \emph{uniformly finitely precomplete} (abbreviated
as \emph{u.f.p.})~if there exists a computable function of three variables
$f(D,e,x)$ (where $D$ is a finite set given by its canonical index) such that
\[
(\forall D,e,x)[\phi_e(x) \downarrow \,\&\, (\exists y)[y \in D \,\&\,
\phi_e(x) \rel{S} y] \Rightarrow \phi_e(x)\rel{S} f(D,e,x)].
\]
\end{emdef}

Then we have:

\begin{fact}\label{fact:Montagna-caratterizzazione}
For every ceer $S$, $S \simeq_{\mathrm{s}} \sim_T$ if and only if $S$ is
u.f.p.~and possesses a diagonal function.
\end{fact}

The rest of the section is devoted to seeing that there is a non-commutative
and non-Boolean c.e.~ring whose word problem is strongly isomorphic to
$\sim_T$. The following result is essentially a rephrasing of Theorem~4.1
of~\cite{pre-lattices}.

\begin{lemma}\label{lem:from-ei-to-ufp}
Let $A$ be a c.e.~algebra whose type contains two binary operations $+,
\cdot$, and two constants $0,1$ such that $+$ is associative, the pair $(U_0,
U_1)$ is e.i., where
\[
U_i=\{x: x =_A i\},
\]
and, for every $a$,
\[
a+0=_A a, \qquad a\cdot 0=_A 0, \qquad a\cdot 1=_A a.
\]
Then $=_A$ is a u.f.p.~ceer.
\end{lemma}

\begin{proof}
For the convenience of the reader, we recall the argument in
\cite{pre-lattices}, adapting it to our context and notations. We  look for a
computable function $f(D,e,x)$ such that if $\phi_e(x) \downarrow$, and
$\phi_e(x)=_A d$ for some $d \in D$ then $f(D,e,x)=_A \phi_e(x)$.

Let $p$ be a productive function for the pair $(U_0, U_1)$: it is well known
that we may assume that $p$ is total. Let
\[
\{u_{d,D,e, x}, v_{d,D,e,x}: \text{ $D$ finite subset of $\omega$, }
d\in D, e,x \in \omega\}
\]
be a computable set of indices we control by the Recursion Theorem. For a
pair $(u_{d,D,e,x}, v_{d,D,e,x})$ in this set let
$c_{d,D,e,x}=p(u_{d,D,e,x},v_{d,D,e,x})$ and $a_{d,D,e,x}=d \cdot
c_{d,D,e,x}$. Define
\[
f(D,e,x)=\sum_{d \in D}a_{d,D,e,x}.
\]
Let us define two c.e.~sets $W_{u_{d,D,e,x}}$ and $W_{v_{d,D,e,x}}$ for each
$d\in D$, which are computably enumerated as follows. Wait for $\phi_e(x)$ to
converge to some $y$ which is $=_A$ to some element in $D$, and while
waiting, we let $W_{u_{d,D,e,x}}$ and $W_{v_{d,D,e,x}}$ enumerate $U_0$ and
$U_1$, respectively. If we wait forever then for all $d \in D$ we end up with
$W_{u_{d,D,e,x}}=U_0$ and $W_{v_{d,D,e,x}}=U_1$. If the wait terminates, let
$d_0\in D$ be the first seen so that $\phi_e(x)=_A d_0$, enumerate also
$c_{d_0,D,e,x}$ into $W_{u_{d_0,D,e,x}}$: this ends up with
$W_{u_{d_0,D,e,x}}=U_0 \cup \{c_{d_0,D,e,x}\}$ and $W_{v_{d_0,D,e,x}}=U_1$,
thus forcing $c_{d_0,D,e,x} =_A 1$ (since it must be that $W_{u_{d_0,D,e,x}}
\cap W_{v_{d_0,D,e,x}}\ne \emptyset$, for otherwise
$c_{d_0,D,e,x}=p(u_{d_0,D,e,x},v_{d_0,D,e,x}) \in W_{u_{d_0,D,e,x}} \cup
W_{v_{d_0,D,e,x}}$, a contradiction) and thus $a_{d_0,D,e,x}=_A d_0$. For all
$d \in D$ with $d \ne d_0$, we let $W_{u_{d,D,e,x}}=U_0$ and
$W_{v_{d,D,e,x}}=U_1 \cup \{c_{d,D,e,x}\}$: this forces $c_{d,D,e,x} =_A 0$
and thus $a_{d,D,e,x}=_A 0$ for each such $d$. Therefore $f(D,e,x)=\sum_{d
\in D}a_{d,D,e,x} =_A  d_0 =_A \phi_e(x)$.
\end{proof}

In order to prove the existence of a ring with the desired properties, let us
first recall the notion of \emph{free ring}. For more details on the
following construction see for instance  paragraph IV.2 of \cite{Cohn}.

Let $R$ be a ring and $M$ be a monoid. The \emph{monoid ring} of $M$ over
$R$, denoted $RM$, is the set
\[
\left\lbrace \varphi: M \rightarrow R: \ \text{supp}(\varphi)
\ \text{is finite}  \right\rbrace,
\]
where $\text{supp}(\varphi)=\left\lbrace m \in M: \ \varphi(m)\neq 0
\right\rbrace$, equipped with the following operations. Given $\varphi, \psi
\in RM$, their \emph{sum} is the function $\varphi + \psi: M \rightarrow R$
given by
\[
\left(\varphi + \psi \right)(m) = \varphi(m)+\psi(m),
\]
and their \emph{product} is the function $\varphi\psi : M \rightarrow R$
given by
\[
\left(\varphi\psi\right)(m)=\sum_{hk=m}\varphi(h)\psi(k).
\]

\begin{remark} \label{normalform}
Equivalently, as is easily seen, $RM$ is the set of formal sums
\[
\sum_{m \in M} r_mm,
\]
where $r_m \in R$, $m \in M$ and $r_m = 0$ for all but finitely many $m$,
equipped with coefficient-wise sum, and product in which the elements of $R$
commute with the elements of $M$.
\end{remark}

\begin{emdef}
The \emph{free ring} on a set $X$ (denoted $\mathbb{Z}X^*$) is the monoid
ring of the free monoid $X^*$ over the ring $\mathds{Z}$ of the integers.
\end{emdef}

\begin{theorem}\label{thm:ring}
There exist non-commutative and non-Boolean c.e.~rings $R$ satisfying that
$=_R \simeq_{\mathrm{s}} \sim_T$.
\end{theorem}

\begin{proof}
Assume that $X = \left\lbrace x_i : \ i \in \omega \right\rbrace$ is a
decidable set and consider the free ring $R^-=\mathbb{Z}X^*$. Notice that, up
to coding, we can identify the universe of $R^-$ with $\omega$ and assume
that its operations are computable and equality is decidable.

Let $U,V \subseteq \mathbb{N}$ be an e.i.~pair of c.e.~sets, and consider the
ideal $K$ of $\mathbb{Z}X^*$ generated by
\[
\left\lbrace x_i : \ i \in U
\right\rbrace \cup \left\lbrace 1-x_j : \ j \in V
\right\rbrace.
\]
Thus, any element of $K$ is of the form
\begin{equation*} \tag{$\dag$}\label{ideal}
\sum_{i\in I}r_{i}\tau_{i}x_{i}^{U}\rho_{i}
+\sum_{j\in J}s_{j}\mu_{j}(1-x_{j}^{V})\nu_{j}.
\end{equation*}
where each of $r_i, s_j$ is in $\mathbb{Z}$, and each of $\tau_i, \rho_i$,
$\mu_j, \nu_j$ is in $X^*$, and finally $I\subseteq U$ and $J\subseteq V$ are
finite sets. Up to shrinking the sets of indices, we can suppose that no
further simplification can be made in either sum.

The ideal $K$ gives rise to a congruence, which we still denote with $K$,
such that $[0]_K = K$. We claim that $1 \notin K$, which implies
\[
[0]_K \cap [1]_K = \emptyset.
\]
To see that our claim is true, we show in fact that no nonzero integer can be
written as in (\ref{ideal}). Calculating we get
\begin{equation*}\tag{$\dag \dag$}\label{linear}
\sum_{i\in I}r_{i}\tau_{i}x_{i}^{U}\rho_{i} +
\sum_{j \in J_0}s_{j}\mu_{j}\nu_{j}
-\sum_{j \in J_0}s_{j}\mu_{j}x_{j}^{V}\nu_{j}
+\sum_{j \in J_1}s_{j}
-\sum_{j \in J_1}s_{j}x_{j}^{V},
\end{equation*}
where $J_{0}=\{j\in J: \mu_{j}\nu_{j} \neq \lambda\}$ and $J_1=\{j\in J:
\mu_{j}\nu_{j} = \lambda\}$. By our assumptions, neither the first, nor the
third, nor the last sum of (\ref{linear}) contain any pair of like monomials,
so that in these sums no further simplification can be made. In order to get
a nonzero integer $s$ from this sum we must have that
\begin{equation*}\tag{$\ddag$}\label{second}
0=\sum_{i\in I}r_{i}\tau_{i}x_{i}^{U}\rho_{i}
+\sum_{j \in J_0}s_{j}\mu_{j}\nu_{j}-
\sum_{j \in J_0}s_{j}\mu_{j}x_{j}^{V}\nu_{j}-\sum_{j \in J_1}s_{j}x_{j}^{V},
\end{equation*}
$\sum_{j \in J_1}s_{j}=s$, and $J_1 \ne \emptyset$.

We are going to see that  the assumption $J_{1}\ne \emptyset$ leads to a
contradiction, by showing that there would be an infinite sequence
$\alpha_{n}=s\sigma_{n}$ ($n \ge 1$), with $\sigma_{n}\in  \{x_{j}: j \in
V\}^{*}$ of length $n$, and $s \in \mathbb{Z}\smallsetminus \{0\}$, such that
each $\alpha_{n}$ occurs as an summand in the second sum of (\ref{second}).

Take $j \in J_{0}$ and let $s=s_j$. So $-s_jx_{j}^{V}$ occurs in the fourth
sum of (\ref{second}). To cancel the monomial $s_{j}x_{j}^{V}$ in
(\ref{second}), there must be a monomial of the form
$s\mu_{j_{1}}\nu_{j_{1}}$ (hence from the second sum) such that
$\mu_{j_{1}}\nu_{j_{1}}=x^{V}_{j}$. Let $\sigma_{1}=x^{V}_{j}$, and
$\alpha_{1}=s \sigma_{1}$. So $\alpha_{1}$ satisfies the claim. Now suppose
that we have found already $\alpha_{n}=s \sigma_{n}$ in the second sum and
satisfying the claim. Then $\sigma_{n}$ is of the form
$\mu_{j_{n}}\nu_{j_{n}}$ which (via multiplication
$s\mu_{j_{n}}(1-x_{j_{n}})\nu_{j_{n}}$ in (\ref{ideal})) corresponds to an
summand in the third sum $-s\mu_{j_{n}}x_{j_{n}}\nu_{j_{n}}$, so that
$\sigma_{n+1}= \mu_{j_{n}}x_{j_{n}}\nu_{j_{n}}$ has length $n+1$, and lies in
$\{x_{j}: j \in V\}^{*}$. Again, this cannot cancel with anything in the
first sum, for each summand in the first sum contains an element indexed from
$U$; it cannot cancel with anything in the fourth sum, nor can it cancel with
anything in the third sum, because we have assumed that it does not contain
like monomials; so it must cancel with something in the second sum, which
therefore contains  $\alpha_{n+1}=s\sigma_{n+1}$ satisfying the  claim.

\begin{lemma}
$(U,V) \leq_m ([0]_J, [1]_J)$.
\end{lemma}

\begin{proof}
We want to show that $(U,V) \leq_m ([0]_J, [1]_J)$ via $f(i)=x_i$. Thus we
must verify that
\[
x_i \in [0]_K \Leftrightarrow i \in U,
\]
and
\[
x_j \in [1]_K \Leftrightarrow j \in V.
\]

The facts that $i \in U$ implies $x_i \in [0]_K$ and $j \in V$ implies $x_j
\in [1]_{K}$ are obvious.

On the other hand, if $x_{i} \in [0]_K$, then $x_{i}$ must be of the form
(\ref{ideal}), from which we obtain again the expression (\ref{linear}), with
the same assumptions on already done simplifications. Assume that there is an
$x_i \in [0]_k$ with $i \notin U$. Since no nonzero integer must appear,
either $J_1 = \emptyset$ or $J_1$ has at least two elements. Assume the
latter. Then in the last sum there is a monomial $-s_j x^V_j$ which must
cancel with a like monomial, which can be nowhere but in the second sum. But
the existence of such a monomial implies that there is a monomial of the form
$s_j x^V_j x^V_{j'}$ or $s_j x^V_{j'} x^V-j$ which in turn leads to a
contradiction, by an argument similar to the one above. Thus $J_1$ must be
empty. Now assume $J_0$ is non empty, so that there is $j \in J_0$ with
$\mu_j \nu_j = x_i$, where $i \notin U$. But then in the third sum there must
be a corresponding monomial ($-x_i x^V_j$ or $-x^V_j x_i$), whose existence,
by reasoning as in the argument used to see that non nonzero integer lies in
$K$, leads again to a contradiction.

Since $x_j \in [1]_K$ if and only if $1-x_j \in [0]_K$, a completely similar
argument shows that $x_j \in [1]_K$ implies $j \in V$.
\end{proof}

Consider the ring $R$ obtained by dividing $R^-$ by the congruence $K$. $R$
is a c.e.~ring according to Definition~\ref{def:c.e.algebras}, as it can be
positively presented as $\langle \omega, F, E\rangle$ where we effectively
identify modulo coding $R^{-}$ with $\omega$, $F$ is the set of computable
operations on $\omega$ which correspond via coding to the operations of
$R^{-}$, and $E$ is the ceer induced on $\omega$ by the congruence $K$.

Moreover, $R$ is equipped with two binary operations $+, \cdot$ (which are
its ring binary operations) and two constants $0,1$ (again, its ring zero-ary
operations). Therefore $=_R$ is a u.f.p.~ceer by
Lemma~\ref{lem:from-ei-to-ufp}. To conclude that $=_R$ is strongly isomorphic
to $\sim_T$ is then enough by Fact~\ref{fact:Montagna-caratterizzazione} that
we find a diagonal function for $=_R$. For this, just take any $v \ne_R 0$,
and consider the function $d(u)=u+v$. It immediately follows that $d(u) \ne_R
u$, for otherwise $v=_R 0$.
\end{proof}


\end{document}